\documentclass[11pt]{amsart}
\usepackage{amsmath,amsthm,amsfonts,amssymb,latexsym,epsfig}

\newtheorem{theorem}{Theorem}[section]

\newtheorem{lemma}{Lemma}[section]

\numberwithin{equation}{section}

\theoremstyle{definition}

\theoremstyle{remark}

\begin{document}
\title{On weighted Hardy inequalities for non-increasing sequences}
\author{Peng Gao}
\address{Department of Mathematics, School of Mathematics and System Sciences, Beijing University of Aeronautics and Astronautics, P. R. China}
\email{penggao@buaa.edu.cn}
\thanks{The author is supported in part by NSFC grant 11371043.}

\begin{abstract}
  A result of Bennett and Grosse-Erdmann
characterizes the weights for which the corresponding weighted Hardy inequality holds on the cone of non-negative, non-increasing sequences and a bound for the best constant is given. In this paper, we improve the bound for $1<p \leq 2$.
\end{abstract}

\maketitle
\section{Introduction}
\label{sec 1} \setcounter{equation}{0}

  Throughout this paper, we let $p \geq 1$. For $p \neq 1$  we let $q$ be defined by $\frac{1}{p}+\frac{1}{q}=1$ and we set $1/q=0$ when $p=1$. Consider the following weighted Hardy inequality on the cone of non-negative, non-increasing sequences ${\bf x}=(x_n)_{n \geq 1}$:
\begin{align}
\label{1.01}
 \sum^{\infty}_{n=1}b_n \left (\sum^{n}_{k=1}\frac {x_k}{n} \right
   )^p \leq U_p\sum^{\infty}_{n=1}b_nx^p_n,
\end{align}
   where $(b_n)_{n \geq 1}$ is a non-negative sequence, $U_p>0$ a constant independent of ${\bf x}$. In \cite[Theorem 1]{BGE}, Bennett and Grosse-Erdmann gave a complete characterization on the sequence $(b_n)_{n \geq 1}$ such that $U_p$ exists. They showed that this is the case if and only if there exists a constant $U'_p>0$ such that for all $n \geq 1$,
\begin{equation*}
   \sum^{\infty}_{k=n}\frac {b_k}{k^p} \leq \frac {U'_p}{n^p}\sum^{n}_{k=1}b_k.
\end{equation*}
   Moreover, if the constants $U_p, U'_p$ are chosen best possible, then
\begin{align}
\label{1.2}
   U'_p \leq U_p \leq p^p(U'_p+1)^p.
\end{align}

   Integral inequalities analogous to \eqref{1.01} for non-increasing functions have been studied by Ari\~no and Muckenhoupt in \cite{AM}. They showed that if $p \geq 1$ and $v$ is a non-negative measurable
function on $(0, \infty)$ then there is a constant $V_p>0$ such that
\begin{align*}
  \int^{\infty}_0v(x)\left ( \frac 1x\int^x_0f(t)dt\right )^p dx \leq V_p\int^{\infty}_{0}v(x)f^p(x)dx
\end{align*}
   holds for all non-negative non-increasing functions $f(x)$ if and only if there is a constant $V'_p>0$ such that for all $x>0$,
\begin{align*}
  \int^{\infty}_x\frac{v(t)}{t^p}dt \leq \frac {V'_p}{x^p}\int^{x}_{0}v(t)dt.
\end{align*}
   The argument of Bennett and Grosse-Erdmann also works for the integral case and it implies that \cite[(17)]{BGE} if the constants $V_p, V'_p$ are chosen best possible, then
\begin{align*}
   V'_p \leq V_p \leq (V'_p+1)^p.
\end{align*}

   Comparing the above two results, we see that in the discrete case, the corresponding bounds for the best constants are not as good as what is given in the integral case. It is then natural to seek for an improvement on the bounds given in \eqref{1.2}, which is the goal of this paper.
   Our result in this paper is the following generalization of the above mentioned result of Bennett and Grosse-Erdmann:
\begin{theorem}
\label{thm1}
   Let $p \geq 1$ be fixed. Let $(b_n)_{n \geq 1}$ be a non-negative sequence and let $(\lambda_n)_{n \geq 1}$ be a non-negative,
    non-increasing sequence with $\lambda_1 >0$. Let $\Lambda_n=\sum^n_{k=1}\lambda_k$.
   Then there is a constant $U_p > 0$ such that
\begin{align}
\label{1.0}
   \sum^{\infty}_{n=1}b_n \left (\sum^{n}_{k=1}\frac {\lambda_kx_k}{\Lambda_n} \right
   )^p \leq U_p\sum^{\infty}_{n=1}b_nx^p_n
\end{align}
  holds for all non-negative, non-increasing sequences $(x_n)_{n \geq 1}$ if and
only if there is a constant $U'_p > 0$ such that for all $n \geq 1$,
\begin{align}
\label{1.0'}
   \sum^{\infty}_{k=n}\frac{b_k}{\Lambda^p_k} \leq \frac
   {U'_p}{\Lambda^p_n}\sum^{n}_{k=1}b_k.
\end{align}
   Moreover, if $U_p$ and $U'_p$ are chosen best-possible then we have
\begin{align}
\label{1.7}
   U'_p \leq U_p \leq  \left\{\begin{array}{ll}
\left (pU'_p+1 \right )^p, & 1 \leq p \leq 2;   \\
 p^p(U'_p+1)^p, & p >2.
\end{array}\right.
\end{align}
\end{theorem}

   The case $\lambda_n=1$ of Theorem \ref{thm1} gives back the result of
   Bennett and Grosse-Erdmann except that instead of \eqref{1.7}, the upper
   bound given for $U_p$ in \cite[Theorem 1]{BGE} is given as in \eqref{1.2} for all $p \geq 1$. Theorem \ref{thm1} therefore improves upon the result of Bennett and Grosse-Erdmann for $1 < p \leq 2$ in this sense.
We point out here that this improvement comes from our refinement (see Lemma \ref{lem1}) on the so called ``Power Rule" (Lemma \ref{lem0} below), a key lemma used in the proof of \cite[Theorem 1]{BGE} by Bennett and Grosse-Erdmann.

\section{lemmas}
\label{sec 2} \setcounter{equation}{0}
   
\begin{lemma}[{\cite[Lemma 3]{BGE}}]
\label{lem0} Let $p \geq 1$. Then for all non-negative sequences $(a_k)_{k \geq 1}$, any integer $n \geq 1$,
\begin{align*}
   \left ( \sum^{\infty}_{k=n}a_k\right )^p \leq p\sum^{\infty}_{k=n}a_k\left ( \sum^{\infty}_{i=k} a_i\right
   )^{p-1}.
\end{align*}
\end{lemma}

\begin{lemma}[{{\cite[Lemma 2]{BGE}}}]
\label{lem2}
    Let $(u_n)_{n \geq 1}, (v_n)_{n \geq 1}$ be two non-negative sequences satisfying for any integer $n \geq 1$,
\begin{equation*}
   \sum_{i=1}^n u_i \leq \sum_{i=1}^n v_i,
\end{equation*}
    then for all non-negative, non-increasing sequences $(a_n)_{n \geq 1}$,
\begin{equation*}
   \sum_{i=1}^n u_i a_i \leq \sum_{i=1}^n v_i a_i.
\end{equation*}
\end{lemma}

\begin{lemma}[{\cite[Lemma 3.1]{G7}}]
\label{lem3}
   Let $(B_n )_{n\geq 1}$ and $( C_n )_{n \geq 1}$ be strictly increasing positive sequences with
   $B_1/B_2 \leq C_1 / C_2$. If for any integer $n \geq 1$,
\begin{equation*}
  \frac {B_{n+1}-B_n}{B_{n+2}-B_{n+1}} \leq  \frac
  {C_{n+1}-C_n}{C_{n+2}-C_{n+1}}.
\end{equation*}
  Then $B_{n}/B_{n+1} \leq C_{n} / C_{n+1}$ for any integer $n \geq 1$.
\end{lemma}
\begin{lemma}
\label{lem2.2} Let $1 \leq p \leq 2$ and let $n \geq 1$ be a fixed integer. Let $\lambda=(\lambda_k)_{1 \leq k \leq n}$ be a non-negative, non-increasing sequence with $\lambda_1>0$. For $1 \leq k \leq n$, let $\Lambda_k=\sum^k_{i=1}\lambda_i$ and
\begin{equation*}
    C_{k,p,\lambda}=\frac
   {\Lambda^p_k}{\sum^{k}_{i=1}\lambda_i\Lambda^{p-1}_i}.
\end{equation*}
Then the sequence $(C_{k,p,\lambda})_{1 \leq k \leq n}$ is increasing with respect
to $k$.
\end{lemma}
\begin{proof}
   The assertion holds trivially when $p=1$, so we may assume $p>1$. We may assume $n \geq 2$ and $\lambda_k>0$ for all $1 \leq k \leq n$. We extend the sequence $\lambda$ to be indexed by all positive integers by defining $\lambda_i=\lambda_n/i$ for $i \geq n+1$. We define similarly $\Lambda_k, C_{k,p,\lambda}$ for $k > n$. It therefore suffices to show that $C_{k, p, \lambda} \leq C_{k+1, p, \lambda}$
   for all $k \geq 1$. Applying Lemma \ref{lem3} with $B_k=\Lambda^p_k, C_k=\sum^{k}_{i=1}\lambda_i\Lambda^{p-1}_i$, 
   ones checks directly that $B_1/B_2 \leq C_1/C_2$. Thus, it remains to
   show for that all $k \geq 1$,
\begin{align*}
   \frac {\Lambda^p_{k+1}-\Lambda^p_{k}}{\lambda_{k+1}\Lambda^{p-1}_{k+1}}
   \leq \frac
   {\Lambda^p_{k+2}-\Lambda^p_{k+1}}{\lambda_{k+2}\Lambda^{p-1}_{k+2}}.
\end{align*}
   When we regard $\lambda_{k+2}$ as a variable with $0 \leq
   \lambda_{k+2} \leq \lambda_{k+1}$, then it is easy to see that
   the right-hand side expression above is a decreasing function of
   $\lambda_{k+2}$ and hence it suffices to show that the above
   inequality holds with $\lambda_{k+2}=\lambda_{k+1}$. In this
   case, on setting $\lambda_{k+1}=x, \Lambda_k=y$ with $y \geq
   x$, we can recast the above inequality as
\begin{align*}
   x-(x+y)^p(2x+y)^{1-p}+y^p(x+y)^{1-p} \geq 0.
\end{align*}
   We further set $z=x/y$ to recast the above inequality as
\begin{align*}
   z-(1+z)^p(1+2z)^{1-p}+(1+z)^{1-p} \geq 0.
\end{align*}
   Upon dividing $1+z$ on both sides of the above inequality and
   setting $t=z/(1+z)$, we see that it suffices to show for $0
   \leq t \leq 1/2$,
\begin{align*}
   g(t):=t-(1+t)^{1-p}+(1-t)^p \geq 0.
\end{align*}
   It's easy to see that $g(0)=g'(0)=0$ and
   $g''(t)=p(p-1)((1-t)^{p-2}-(1+t)^{-p-1}) \geq 0$ when $1<p \leq 2$. This implies that $g(t)$ is an increasing function of $0 \leq t \leq 1/2$ which 
   completes the proof.
\end{proof}
\begin{lemma}
\label{lem1} Let $p \geq 1$, $\lambda=(\lambda_k)_{k \geq 1}$ a
non-negative, non-increasing sequence with $\lambda_1>0$.  Then for all non-negative,
non-increasing sequences $(a_k)_{k \geq 1}$, any integer $n \geq 1$,
\begin{equation}
\label{2.1}
   \left ( \sum^n_{k=1}\lambda_ka_k\right )^p \leq C_{n,p, \lambda}\sum^n_{k=1}\lambda_ka_k\left ( \sum^k_{i=1}\lambda_ia_i\right
   )^{p-1},
\end{equation}
   where $C_{n, p,\lambda}$ is defined as in Lemma \ref{lem2.2} when $1 \leq p \leq 2$ and $C_{n,p,\lambda}=p$ when $p>2$. Moreover, when $1 \leq p \leq 2$, the constant $C_{n,p,\lambda}$ is best possible and equality in
    \eqref{2.1} holds when $1<p \leq 2$ if and only if $a_1=a_2=\ldots=a_n$.
\end{lemma}
\begin{proof}
    As inequality \eqref{2.1} follows from Lemma \ref{lem0} when $p>2$ and the assertion of the lemma follows trivially for $p=1$, we only need to consider the case $1<p \leq 2$. We define
\begin{equation*}
    f_n(x_1, x_2, \ldots, x_n)=\left ( \sum^n_{k=1}\lambda_kx_k\right )^p -C_{n,p,\lambda}\sum^n_{k=1}\lambda_kx_k\left ( \sum^k_{i=1}\lambda_ix_i\right
   )^{p-1}.
\end{equation*}
    By homogeneity, it suffices to show $f_n \leq 0$ on the compact set $\{ (x_1, \ldots, x_n) | 1 \geq
    x_1 \geq x_2 \geq \ldots \geq x_n \geq 0 \}$. We may assume $\lambda_k>0$ for all
    $k$ here as discarding the zero terms and relabeling will not change the expression.

   As $f_1=0$ holds trivially, we may assume $n \geq 2$ here.  Assume the maximum of $f_n$ is
    attained at some ${\bf x}_0=\left ( \left( {\bf x}_0 \right)_1, \left( {\bf x}_0 \right)_2, \ldots, \left( {\bf x}_0 \right)_n
    \right)$ with $\left( {\bf x}_0 \right)_1 \geq \left( {\bf x}_0 \right)_2 \geq
\ldots \geq \left( {\bf x}_0
    \right)_n$. If $\left( {\bf x}_0
    \right)_{m+1}=0$ for some $1 \leq m <n$, then as $C_{m,p, \lambda} \leq C_{n,p, \lambda}$ by Lemma \ref{lem2.2}, it is easy to
    see that we are reduced to the consideration of $f_m \leq 0$.
    Thus, we may further assume $\left( {\bf x}_0
    \right)_n >0$ here.

     Suppose $\left( {\bf x}_0
    \right)_m >\left( {\bf x}_0
    \right)_{m+1}>0$ for some $1 \leq m <n$.
    In this case we must have $\partial f_n/\partial x_m({\bf x}_0) \geq 0$ since $\partial f_n/\partial x_m({\bf x}_0) < 0$ means
    decreasing the value of $\left( {\bf x}_0
    \right)_m$ will increase the value of
    $f_n$, a contradiction. Similar argument implies that $\partial f_n/\partial x_{m+1}({\bf x}_0) \leq 0$.
    Therefore, we conclude that we have
\begin{align*}
    0 & \leq \frac {1}{\lambda_m}\frac {\partial f_n}{\partial x_m}({\bf x}_0)-\frac {1}{\lambda_{m+1}}\frac {\partial f_n}{\partial x_{m+1}}({\bf
    x}_0)  \\
    &=C_{n,p, \lambda}\left ( \left ( \sum^{m+1}_{i=1}\lambda_i\left( {\bf x}_0
    \right)_i\right
   )^{p-1}-\left ( \sum^m_{i=1}\lambda_i\left( {\bf x}_0
    \right)_i\right
   )^{p-1}-(p-1)\lambda_m\left( {\bf x}_0
    \right)_m\left ( \sum^m_{i=1}\lambda_i\left( {\bf x}_0
    \right)_i\right
   )^{p-2} \right ).
\end{align*}
   If $p=2$, this would imply $\lambda_{m+1}\left( {\bf x}_0
    \right)_{m+1} \geq \lambda_m\left( {\bf x}_0
    \right)_m$, a contradiction. If $1<p<2$, by the Mean Value Theorem, we have
\begin{align*}
    & \left ( \sum^{m+1}_{i=1}\lambda_i\left( {\bf x}_0
    \right)_i\right
   )^{p-1}-\left ( \sum^m_{i=1}\lambda_i\left( {\bf x}_0
    \right)_i\right
   )^{p-1} \\
   =&(p-1)\lambda_{m+1}\left( {\bf x}_0
    \right)_{m+1} \xi^{p-2} <(p-1)\lambda_{m}\left( {\bf x}_0
    \right)_m\left ( \sum^m_{i=1}\lambda_i\left( {\bf x}_0
    \right)_i\right
   )^{p-2},
\end{align*}
   as $\sum^m_{i=1}\lambda_i\left( {\bf x}_0
    \right)_i < \xi < \sum^{m+1}_{i=1}\lambda_i\left( {\bf x}_0
    \right)_i$. This again leads to a contradiction. Thus we must have
   $\left( {\bf x}_0 \right)_1 =\left( {\bf x}_0
\right)_2=\ldots=\left( {\bf x}_0 \right)_n$, which implies that
$f_n({\bf x}_0)=0$ and the assertion
   of the lemma follows for $1<p \leq 2$.
\end{proof}

    In what follows we make two remarks about Lemma \ref{lem1}.
    Throughout our remarks, we let $1 \leq p \leq 2$, $\lambda_k=1$ for all $k$ with the function $f_n$ being defined as in the proof of Lemma
    \ref{lem1} and $C_{n,p,\lambda}$ being defined as in
    Lemma \ref{lem2.2}.

    Remark 1. For any given ${\bf x}=(x_1,
    x_2, \ldots, x_n)$, we let ${\bf x}'=(x_1, x_2, \ldots, x_{i+1}, x_i, \ldots, x_n)$ by permuting
    two adjacent coordinates $x_i, x_{i+1}$ of ${\bf x}$ for some $1 \leq i
    <n$, then we have
\begin{align*}
   &f_n({\bf x})-f_n({\bf x}')  \\
   =&-C_{n,p, \lambda}\left ( x_i(a+x_i)^{p-1}+x_{i+1}(a+x_i+x_{i+1})^{p-1}-x_{i+1}(a+x_{i+1})^{p-1}-x_{i}(a+x_i+x_{i+1})^{p-1} \right
   ),
\end{align*}
   where we set (with empty sum being $0$)
   $a=\sum^{i-1}_{k=1}x_k$.

    It is easy to check that the function
    $S_r(x,y)=(x^r-y^r)/(x-y)$
     is an increasing (respectively, decreasing) function of $y>0$ for fixed $0<x \neq y$ when $r \geq 1$ (respectively, $0<r \leq 1$). Apply this with $r=p-1$, $x=a+x_i+x_{i+1}, y=a+x_i,
     y'=a+x_{i+1}$, we see immediately that $f_n({\bf x}) \geq f_n({\bf
     x}')$ when $x_{i+1} \geq x_i \geq 0$ and $1<p \leq 2$ or when $x_{i} \geq  x_{i+1} \geq 0$ and $p \geq 2$.

     It follows that when $p=2$ and $\lambda_k=1$ for all $k$, the
   maximum of $f_n$ on all non-negative sequences is the same as the maximum of $f_n$ on all non-negative,
non-increasing sequences. Thus, when $p=2,\lambda_k=1$ for all $k$, the assertion of Lemma
\ref{lem1} holds for all non-negative sequences.

     Remark 2.  We also remark that when $p>2$, we have for $n \geq 2$,
\begin{align}
\label{2.2'}
     \frac {\partial f_n}{\partial x_n}\left(\left(1, 1, \ldots, 1 \right)
     \right)=n^{p-2}\left ( np-C_{n,p,\lambda}\left(n+p-1 \right ) \right
     ) <0,
\end{align}
    where the last inequality is
    equivalent to
\begin{equation*}
    \sum^{n}_{k=1}k^{p-1} <  \frac {n^{p-1}}{p}(n+p-1), \ n \geq 2,
\end{equation*}
     which in turn can be easily established by induction.

     Inequality \eqref{2.2'} implies that in this case
    $0=f_n\left(\left(1, 1, \ldots, 1 \right)
     \right)<f_n\left(\left(1, 1, \ldots, 1-\epsilon \right)
     \right)$ for some $\epsilon>0$ small enough and this shows
     that inequality \eqref{2.1} does not hold for
     all non-negative, non-increasing sequences when $p>2$.

\section{Proof of Theorem \ref{thm1}}
\label{sec 3} \setcounter{equation}{0}
    We now proceed to the proof of Theorem \ref{thm1}. Our
    approach here follows that of Bennett and Grosse-Erdmann in their proof of \cite[Theorem
1]{BGE}. By considering the sequences $(1, \ldots , 1,$$ 0, 0, $$
\ldots )$,  we see first that \eqref{1.0'} is a necessary condition for
the validity of inequality \eqref{1.0} and that $U'_p \leq U_p$.
Conversely, assume that condition \eqref{1.0'} holds. Note first that it follows from Lemma \ref{lem0} and \ref{lem1} that $C_{n,p,\lambda} \leq p$ where $C_{n,p,\lambda}$ is defined as in Lemma \ref{lem1}. Further
note that for any integer $n \geq 1$,
\begin{align}
\label{2.3}
   \sum^n_{k=1} \lambda_k\Lambda^{p-1}_k\sum^{\infty}_{i=k}C_{i,p, \lambda}\frac {b_i}{\Lambda^p_i}
   & \leq \sum^n_{k=1} \lambda_k\Lambda^{p-1}_k\sum^{n}_{i=k}C_{i,p,
\lambda}\frac {b_i}{\Lambda^p_i}
   +\sum^n_{k=1} \lambda_k\Lambda^{p-1}_k\sum^{\infty}_{i=n}C_{i,p, \lambda}\frac {b_i}{\Lambda^p_i} \\
  & \leq \sum^{n}_{i=1}C_{i,p, \lambda}\frac {b_i}{\Lambda^p_i}\sum^i_{k=1}
  \lambda_k\Lambda^{p-1}_k+p\sum^n_{k=1} \lambda_k\Lambda^{p-1}_k\sum^{\infty}_{i=n}\frac
  {b_i}{\Lambda^p_i} \nonumber \\
  & \leq \sum^{n}_{i=1}C_{i,p, \lambda}\frac {b_i}{\Lambda^p_i}\sum^i_{k=1}
  \lambda_k\Lambda^{p-1}_k+pU'_p\frac {1}{\Lambda^p_n}\sum^n_{k=1}
  \lambda_k\Lambda^{p-1}_k\sum^{n}_{i=1}b_i \nonumber \\
  & \leq  U''_p\sum^n_{i=1}b_i, \nonumber
\end{align}
   where $U''_p=pU'_p+1, 1 \leq p \leq 2, U''_p=pU'_p+p, p>2$ and we have used \eqref{1.0'} in the third inequality above and the bound $\sum^n_{k=1}
  \lambda_k\Lambda^{p-1}_k \leq \sum^n_{k=1}
  \lambda_k\Lambda^{p-1}_n=\Lambda^p_n$ in the last inequality above.

 Now by Lemma \ref{lem1}, we have, for any non-negative, non-increasing sequences $(x_n)_{n \geq 1}$,
\begin{align*}
   \sum^{\infty}_{n=1}b_n \left (\sum^{n}_{k=1}\frac {\lambda_kx_k}{\Lambda_n} \right
   )^p & \leq \sum^{\infty}_{n=1}C_{n,p, \lambda}\frac {b_n}{\Lambda^p_n}\sum^{n}_{k=1}\lambda_kx_k\left (\sum^{k}_{i=1}\lambda_ix_i \right
   )^{p-1}  \\
   &=\sum^{\infty}_{k=1}\lambda_kx_k\left(  \sum^{\infty}_{n=k}C_{n,p, \lambda}\frac {b_n}{\Lambda^p_n} \right )\left (\sum^{k}_{i=1}\lambda_ix_i \right
   )^{p-1} \\
   &= \sum^{\infty}_{k=1}\left( \lambda_k\Lambda^{p-1}_k\sum^{\infty}_{n=k}C_{n,p, \lambda}\frac {b_n}{\Lambda^p_n} \right )x_k
   \left (\sum^{k}_{i=1}\frac {\lambda_ix_i}{\Lambda_k} \right
   )^{p-1}  \\
   & \leq U''_p\sum^{\infty}_{k=1}b_kx_k
   \left (\sum^{k}_{i=1}\frac {\lambda_ix_i}{\Lambda_k} \right
   )^{p-1}  \\
   &= U''_p\sum^{\infty}_{k=1}b^{1/p}_kx_kb^{1/q}_k
   \left (\sum^{k}_{i=1}\frac {\lambda_ix_i}{\Lambda_k} \right
   )^{p-1},
\end{align*}
    where the second inequality above follows from Lemma \ref{lem2} and \eqref{2.3}, the sequence
    $$\left(x_k
   \left (\sum^{k}_{i=1}\frac {\lambda_ix_i}{\Lambda_k} \right
   )^{p-1} \right)_{k \geq 1}$$ being non-negative, non-increasing.

     By H\"older's inequality, we then have
\begin{align*}
   \sum^{\infty}_{n=1}b_n \left (\sum^{n}_{k=1}\frac {\lambda_kx_k}{\Lambda_n} \right
   )^p & \leq  U''_p\left (\sum^{\infty}_{n=1}b_nx^p_n \right )^p\left
  (\sum^{\infty}_{k=1}b_k
   \left (\sum^{k}_{i=1}\frac {\lambda_ix_i}{\Lambda_k} \right
   )^p\right )^{1/q},
\end{align*}
    which implies \eqref{1.0} with $U_p$ being replaced by ${U''_p}^p$ and this completes the
  proof of Theorem \ref{thm1}.


\end{document}